\documentclass[a4paper, 11 pt, twoside]{amsart}
\usepackage{srollens-en}[2011/03/03]
\usepackage[left = 3.5cm, right = 3.5cm, headsep = 6mm,
footskip = 10mm, top = 35mm, bottom = 35mm, footnotesep=5mm, headheight =
2cm]{geometry}

\usepackage[usenames, dvipsnames]{xcolor}
\usepackage[T1]{fontenc}

\usepackage{caption}
\usepackage{subcaption}
\DeclareCaptionSubType*[Alph]{figure}
\usepackage{tikz, tikz-cd}
\usetikzlibrary{through}
\usetikzlibrary{decorations.pathreplacing,decorations.markings}
\tikzset{
  on each segment/.style={
    decorate,
    decoration={
      show path construction,
      moveto code={},
      lineto code={
        \path [#1]
        (\tikzinputsegmentfirst) -- (\tikzinputsegmentlast);
      },
      curveto code={
        \path [#1] (\tikzinputsegmentfirst)
        .. controls
        (\tikzinputsegmentsupporta) and (\tikzinputsegmentsupportb)
        ..
        (\tikzinputsegmentlast);
      },
      closepath code={
        \path [#1]
        (\tikzinputsegmentfirst) -- (\tikzinputsegmentlast);
      },
    },
  },
  mid arrow/.style={postaction={decorate,decoration={
        markings,
        mark=at position .5 with {\arrow[#1]{stealth}}
      }}},
}

\usepackage[unicode,bookmarks, pdftex]{hyperref}
\hypersetup{colorlinks=true,citecolor=NavyBlue,linkcolor=NavyBlue,urlcolor=Orange, pdfpagemode=UseNone, breaklinks=true}

{\theoremstyle{Lehn-up}
\newtheorem{Construction}[equation]{Construction}
}

\newcommand{\Diff}{\mathrm{Diff}}

\title[Stable Godeaux surfaces]{A new  irreducible component of the moduli space of stable Godeaux surfaces}

\author{S\"onke Rollenske}
\address{S\"onke Rollenske\\Fakult\"at f\"ur Mathematik\\Universt\"at Bielefeld\\Universit\"atsstr. 25\\33615 Bielefeld\\Germany}
\email{rollenske@math.uni-bielefeld.de}

\begin{document}
\begin{abstract}
We construct from a general del Pezzo surface of degree 1 a Gorenstein stable surfaces $X$ with $K_X^2=1$ and $p_g(X)=q(X)=0$. These surfaces are not smoothable but give an open subset of an irreducible component of the moduli space of stable Godeaux surfaces. 

In a particular example we also compute the canonical ring explicitly and discuss the behaviour of  pluricanonical maps.
\end{abstract}

\maketitle

\setcounter{tocdepth}{1}
\tableofcontents

\section{Introduction}

One of the most vexing problems in the classification of surfaces of general type is the fact that we are not yet able to classify surfaces with the smallest possible invariants $K_X^2 = 1$ and $p_g(X)=q(X)=0$. Such surfaces are usually called  (numerical) Godeaux surfaces \cite{BHPV}.

Nowadays the moduli space of surfaces of general type comes with a modular compactification, the moduli space of stable surfaces \cite{kollar12, KollarModuli}. The stable surfaces occurring in the compactification can be used to gain insight in the classical moduli space but are also interesting in their on right.

This article grew out of two motivations: first of all we hoped to give a new construction of some   Godeaux surfaces by smoothing stable Godeaux surfaces in the spirit of \cite{lee-park07} but starting with a Gorenstein surface. Secondly,   in our study of pluricanonical maps of stable surfaces \cite{liu-rollenske14} we were looking for examples of Gorenstein stable surfaces with  $|5K_X|$ not very ample; so it was natural to study some examples with small invariants. 

Both motivations lead us to the construction described in Section \ref{sect: surface} where from any general smooth del Pezzo surface of degree 1 we construct a (Gorenstein) stable  Godeaux surface, that is a Gorenstein stable surface $X$ with $K_X^2=1$ and $p_g(X)=q(X)=0$. We also show that $X$ is topologically simply connected and compute its integral homology.

It turns out that our first hope was futile: the surfaces we construct are far away from any smooth surfaces.
\begin{custom}[Theorem A]
 Let $X$ be a stable  Godeaux surface of the type constructed in Section \ref{sect: surface}. Then the Kuranishi space of $X$ is smooth of dimension 8 and all deformations are locally trivial. Such surfaces form an irreducible open subset of the moduli space $\overline{\gothM}_{1,1}$ of stable surface with $K_X^2=\chi(\ko_X)=1$, which has at most finite quotient singularities. 
\end{custom}
The proof of Theorem A will be given in Section \ref{sect: proof}. The key point is the explicit control over  infinitesimal deformations in Section \ref{sect: deformations}, where we have to overcome the annoying fact that we are handling local normal crossing surfaces which do not have global normal crossings.

In Section \ref{sect: 3} we will concentrate on one explicit example and compute, with the help of a computer algebra system, its  canonical ring. The result is a rather complicated ring, Gorenstein in codimension ten, but it is generated in degree up to 5 and, in particular, $|5K_X|$ is very ample. In fact, one can show directly that the latter property is true for all surfaces in our family \cite{franciosi-rollenske14}.

In principle, it would be possible to study degenerations of our examples via the del Pezzo surfaces, however it is unclear if in this way one would be able to connect to a known component of the moduli space of Godeaux surfaces. 

Starting from the results in \cite{fpr14} one can actually classify Gorenstein stable Godeaux surfaces with worse than canonical singularities. These will be studied in a forthcoming joint paper with Marco Franciosi and Rita Pardini. 

\subsubsection*{Acknowledgements}
 We are indebted to Marco Franciosi, Rita Pardini, and Wenfei Liu for many discussions on stable surfaces. We also enjoyed discussing different aspects of this project with Olivier Benoist, Anne Fr\"ubis-Kr\"uger,  Roberto Pignatelli, Miles Reid, Helge Ruddat, and  Timo Sch\"urg.

The work of the author is supported by the DFG through the Emmy Noether program and SFB 701. Part of this research was carried out while visiting  the University of Pisa, supported by GNSAGA of INDAM, and during a stay at HIM in Bonn.

\section{The surface}\label{sect: surface}
\begin{Construction}\label{Construction}
 Let $\bar X$ be a smooth del Pezzo surface of degree 1, that is, a smooth projective surface with   $-K_{\bar X}$ ample and $K_{\bar X}^2 = 1$. Assume that there are two different curves $\bar D_1, \bar D_2 \in |-K_{\bar X}|$ such that $\bar D_i$ is isomorphic to a plane nodal cubic. We will see below that this condition is satisfied for a general choice of $\bar X$. Let $\bar D=\bar D_1\cup\bar D_2$ and let $\bar \nu\colon \bar D^\nu \to \bar D$ be the normalisation map.
 
 We specify an involution $\tau$ on $\bar D^\nu$, which preserves the preimages of the nodes, and construct a surface $X$ as the pushout in the diagram
 \begin{equation}\label{diagram: pushout}
\begin{tikzcd}
    \bar X \dar{\pi}\rar[hookleftarrow]{\bar\iota} & \bar D\dar{\pi} & \bar D^\nu \lar[swap]{\bar\nu}\dar{/\tau}
    \\
X\rar[hookleftarrow]{\iota}  &D &D^\nu\lar[swap]{\nu}
    \end{tikzcd}.
 \end{equation}
Since $\bar D^\nu$ is the disjoint union of two copies of $\IP^1$, each with three marked points that are preimages of nodes of $\bar D$, the involution $\tau$  is uniquely determined by its action on these  points. This and some further notation is given in Figure \ref{fig: construction}.
\end{Construction}
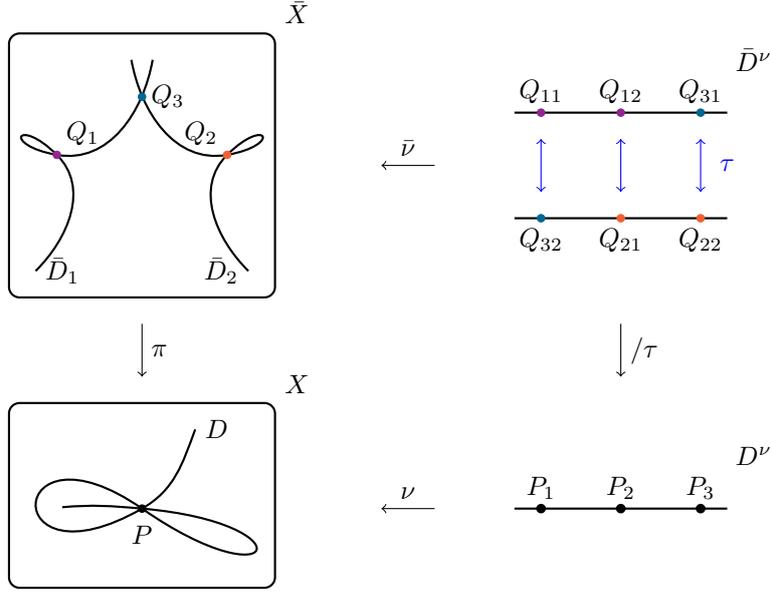
\begin{figure}
\small
\begin{tikzpicture}
[curve/.style ={thick, every loop/.style={looseness=10, min distance=30}},
q1/.style={color=Plum},
q2/.style={color=RedOrange},
q3/.style={color=MidnightBlue},
scale = 0.7
]

\node (Q3) at (0,2.3){};
\node (Q2) at (1.6,1.2) {};
\node (Q1) at (-1.6,1.2) {};

\begin{scope}
\draw[thick, rounded corners] (-2.5, -1.5) rectangle (2.5, 3.5) node [above right] {$\bar X$};
\draw [curve] (-2,-1) node [ right] {$\bar D_1$} to[out=45, in=315] (-1.6, 1.2) to[out=135, in =170,loop] () to[out=350, in = 260] (.2,3);
\draw [curve, xscale=-1] (-2,-1)  node [left] {$\bar D_2$} to[out=45, in=315] (-1.6, 1.2) to[out=135, in =170,loop] () to[out=350, in = 260] (.2,3);

\filldraw[q3] (Q3) circle (2pt) ; \node at (Q3) [right] {$Q_3$};
\filldraw[q2] (Q2) circle (2pt);  \node at (Q2) [above left] {$Q_2$};
\filldraw[q1] (Q1) circle (2pt); \node at (Q1) [above right] {$Q_1$};

\draw[->] (0,-2) to node[right] {$\pi$}++(0,-1);
\end{scope}

\begin{scope}[yshift=-5.5cm]
\draw[thick, rounded corners] (-2.5, -1.5) rectangle (2.5, 2) node [above right] {$X$};
\draw[curve,  every loop/.style={looseness=30, min distance=90}] (1,1.5) node [right]{$D$} to[out = 250, in = 30]
(0,0) to[out = 210, in = 145, loop]
() to[out = 325, in =355 ,  loop]
() to[out =175, in = 5] (179:1.5cm);

\filldraw (0,0) circle (2pt) node [below, yshift = -.1cm] {$P$};
\end{scope}

\begin{scope}[xshift=8cm, yshift=1cm]
\draw [curve] (3,1) -- (2.5,1) node[above] {$Q_{31}$} -- (1,1)  node[above] {$Q_{12}$} to (-0.5,1)  node[above] {$Q_{11}$}--(-1,1);
\filldraw [q3] (2.5,1) circle (2pt);
\filldraw[q1] (1,1) circle (2pt);
\filldraw[q1] (-0.5,1) circle (2pt);

\filldraw [curve] (3,-1)--(2.5,-1)  node[below] {$Q_{22}$}--(1,-1)  node[below] {$Q_{21}$}--(-0.5,-1)  node[below] {$Q_{32}$}--(-1,-1);
\filldraw [q2] (2.5,-1) circle (2pt);
\filldraw[q2] (1,-1) circle (2pt);
\filldraw[q3] (-0.5,-1) circle (2pt);

\node at (3.5, 2) {$\bar D^\nu$};

\begin{scope}[color= blue]
\draw[<->] (2.5, 0.5) -- (2.5, -0.5);
\draw[<->] (1, 0.5) -- (1, -0.5);
\draw[<->] (-.5, 0.5) -- (-.5, -0.5);
\node at (3,0) {$\tau$};
\end{scope}

\draw[->] (-2.5,0) to node[above] {$\bar\nu$}++(-1,0);
\draw[->] (1,-3) to node[right] {$/\tau$}++(0,-1);
\end{scope}

\begin{scope}[xshift=8cm, yshift = -5.5cm]
\filldraw [curve] (3,0) --(2.5,0) circle (2pt) node[above] {$P_{3}$}--(1,0) circle (2pt) node[above] {$P_2$}--(-0.5,0) circle (2pt) node[above] {$P_1$}--(-1,0);

\node at (3.5, 1) {$ D^\nu$};
\draw[->] (-2.5,0) to node[above] {$\nu$}++(-1,0);
\end{scope}

\end{tikzpicture}
\caption{The construction of the surface}\label{fig: construction}
\end{figure}
 Note that  $\pi^*K_X = K_{\bar X}+\bar D=-K_{\bar X}$, which is ample. By \cite[Thm.\ 5.13]{KollarSMMP} the surface $X$ exists as a projective scheme,  has semi-log-canonical singularities and $K_X$ is ample. In fact, in our case it has a degenerate cusp singularity at $P$, locally isomorphic to the origin of $\{xyz=0\}\subset \IC^3$, normal crossings along $D\setminus \{P\}$ and is smooth  outside $D$.

\begin{prop}\label{prop: invariants}
 The surface $X$ has semi-log-canonical, Gorenstein singularities and $K_X$ is ample. It has invariants
 \[ K_X^2 = 1, \chi(\ko_X) = 1, p_q(X) = q(X)=0  \]
and thus is a stable  Godeaux surface. It is topologically simply connected and has  integral homology
\begin{gather*}
H_{1}(X, \IZ) =0,\qquad H_3(X, \IZ)\isom \IZ,\\
 H_{2i}(X, \IZ) =H_{2i}(\bar X, \IZ)= \begin{cases}\IZ & i=0,2\\ \IZ^9 & i=1                
                 \end{cases}.
\end{gather*}

\end{prop}
\begin{proof}
 The statement on singularities and ampleness was proved above and we have  
 \[K_X^2 = (K_{\bar X}+\bar D)^2=(-K_{\bar X})^2=1.\]
 As in  \cite[Prop.~14]{liu-rollenske14} we also have 
 \[\chi(\ko_{X})=\chi(\ko_{\bar X})+\chi(\ko_{D})-\chi(\ko_{\bar D})=1-1+1=1\]
 and thus it remains to show that $p_g(X) = h^2(\ko_X)=h^0(\omega_X)=0$. So assume we have a non-zero effective divisor $C\in |K_X|$. Then $\pi^*D\in |K_{\bar X}+\bar D|=|-K_{\bar X}|$ which is an elliptic pencil with base point $Q_3$. Thus $\pi^* D$ contains $Q_3$ which means that $D$ itself contains the degenerate cusp $P$. Consequently $\pi^*D$ contains $Q_1$ and $Q_2$ as well, which is impossible for a divisor in $|-K_{\bar X}|$. Thus $p_g(X) = 0$ and consequently also $q(X) = 0$.  

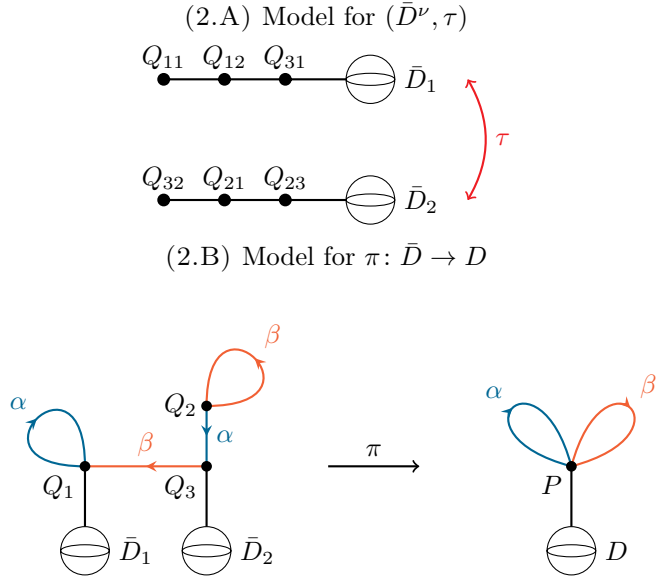
\begin{figure}[t]
\small
\begin{subfigure}{\textwidth}
\centering
\caption{Model for $(\bar D^\nu, \tau)$}\label{fig: bar Dnu }
 \begin{tikzpicture}
[sphere/.style = {thin},
graph/.style = {thick, fill=black, radius = 2.5pt}, 
scale = .8]
\begin{scope}
\draw[graph] (-2,0) circle node[above] {$Q_{11}$}--++ (1,0) circle node[above] {$Q_{12}$}--++ (1,0) circle  node[above] {$Q_{31}$}--++(1,0); 
\draw[sphere] (1.4,0) circle [radius = .4] circle [x radius=.4cm, y radius=.1cm]++(.4,0) node[right] {$\bar D_1$};

\draw[graph] (-2,-2) circle node[above] {$Q_{32}$}--++ (1,0) circle node[above] {$Q_{21}$}--++ (1,0) circle  node[above] {$Q_{23}$}--++(1,0); 
\draw[sphere] (1.4,-2) circle [radius = .4] circle [x radius=.4cm, y radius=.1cm] ++(.4,0) node[right] {$\bar D_2$};;
 
\draw[thick, <->, bend left, Red] (3,0) to node[right]{$\tau$} (3,-2);
 \end{scope}
\end{tikzpicture}
        \end{subfigure}

        \begin{subfigure}{\textwidth}
\caption{Model for $\pi\colon \bar D\to D$ }\label{fig: pi }
\small\centering

\begin{tikzpicture}[sphere/.style = {thin},
graph/.style = {thick},
alpha/.style = {MidnightBlue, thick, postaction={mid arrow}},
beta/.style = {RedOrange,thick, postaction={mid arrow}},
 every loop/.style={looseness=30, min distance=60},
scale = .8,
]

\begin{scope}
\path
(0,0)  coordinate(Q1) ++(2,0) coordinate (Q3) ++(0,1) coordinate(Q2) ;

\draw[sphere] (Q1) ++(0,-1.4) circle [radius = .4] circle [x radius=.4cm, y radius=.1cm]++(.4,0) node[right] {$\bar D_1$};
\draw[thick] (Q1) --++(0,-1);

\draw[sphere] (Q3) ++(0,-1.4) circle [radius = .4] circle [x radius=.4cm, y radius=.1cm]++(.4,0) node[right] {$\bar D_2$};
\draw[thick] (Q3) --++(0,-1);

\draw[alpha] (Q1) to[out=-180, in =90,loop]node[above left] {$\alpha$} ();
\draw[beta] (Q3) to node[above]{$\beta$} (Q1);

\draw[beta] (Q2) to[out=0, in =90,loop]node[above right] {$\beta$} ();
\draw[alpha] (Q2) to node[right]{$\alpha$} (Q3);

\path (8,0) coordinate(P);
\draw[thick] (P) -- ++(0,-1);
\draw[sphere] (P)++(0,-1.4) circle [radius = .4] circle [x radius=.4cm, y radius=.1cm]++(0.4,0) node[right]{$D$};
\draw[beta] (P) to[in=15, out =75,loop] node[above right]{$\beta$} ();
\draw[alpha] (P) to[out=165, in =105,loop] node[above left] {$\alpha$} ();

 \path[fill=black, radius = 2.5pt]
 (Q1) circle   node[below left] {$Q_1$}
(Q2) circle  node[  left] {$Q_2$}
(Q3) circle node[below left] {$Q_3$}
(P) circle node[below left]{$P$} ;
\draw[thick, ->] (4,0) to node[above]{$\pi$} ++(1.5,0);
\end{scope}
\end{tikzpicture}
        \end{subfigure}
        \caption{Homotopy-equivalent models for the construction}\label{fig: topological model}
\end{figure}
For the computation of topological invariants we first choose a suitable homotopy-equivalent models $\pi\colon \bar D\to D$: we replace each component of $\bar D$ by a 2-Sphere with an attached interval on which the glueing takes place; this is shown in Figure \ref{fig: topological model}. In this model, the maps in the Mayer-Vietoris-Sequence in homology (see \cite[Cor.~3.2]{fpr14b})
\[ \begin{tikzcd}
{}\rar& H_i(\bar D, \IZ) \rar{(\pi_*, \bar \iota_*)} & H_i(D, \IZ)\oplus H_i(\bar X, \IZ)\rar{\iota_*-\pi_*} & H_i(X, \IZ)\rar&
\end{tikzcd}\]
are easily computed and give the claimed result. For the fundamental group we have, again by \cite[Cor.~3.2]{fpr14b}, 
\[ \pi_1(X, P) = \frac{\pi_1(D, P)}{\langle\pi_1(\bar D, Q_2)\rangle}=\frac{\langle\alpha, \beta\rangle}{\langle\beta\alpha\inverse\beta, \inverse\alpha\beta\alpha \rangle}={1},\]
because $\bar X$ is simply connected. 
\end{proof}
\begin{rem}
The vanishing of the irregularity can be proved without resorting to the explicit construction given above (see \cite{fpr14b}). 
\end{rem}

\section{The family}\label{sect: family}
Recall  that a smooth del Pezzo surface $\bar X$  of degree 1 is the blow up of $\IP^2$ in  eight points $p_1, \dots, p_8$  in general position \cite{KollarRationalCurves}. On $\IP^2$ the anti-canonical pencil of $X$ induces a pencil of cubics and, conversely, given a (general)  pencil of cubics we can blow up eight of the base points to get a del Pezzo surface of degree 1. For later use we record
\begin{lem}\label{lem: def dP}
 Let $\sigma\colon \bar  X=\mathrm{Bl}_{\{p_1, \dots, p_8\}}(\IP^2)\to \IP^2$ be a smooth del Pezzo surface of degree $1$. Then $h^0(\kt_{\bar X})=h^2(\kt_{\bar X})=0$ and $h^1(\kt_{\bar X})=8$. The local universal deformation space of $X$ is smooth of dimension $8$ and every deformation is induced by a deformation of the points $p_i$ in the plane.
\end{lem}
\begin{proof} 
Let $\ki$ be the ideal sheaf of the points in the plane. Then by \cite[Lem.\ 9.22]{catanese88} $\sigma_*\kt_{\bar X} = \ki\tensor \kt_{\IP^2}$ and there are no higher pushforward sheaves. Since the points are in general position $H^0(\kt_{\IP^2}\tensor \ki)=0$ and the dimensions of the cohomology groups of $\kt_{\bar X}$ follow from the exact sequence 
\[0\to \kt_{\IP^2}\tensor \ki\to \kt_{\IP^2} \to \kt_{\IP^2}\restr{\{p_i\}}\to 0.\]
Since $H^2(\kt_{\bar X})=H^0(\kt_{\bar X}=0$ the surface has a smooth universal deformation of dimension $h^1=\kt_{\bar X})=8$ by \cite[Cor.~2.6.4]{SernesiDeformation}. 
The surjection $H^0(\kt_{\IP^2}\restr{\{p_i\}})\to H^1(\kt_{\bar X})$ is the map that associates to an infinitesimal deformation of the points in $\IP^2$ the corresponding deformation of $X$, which shows the last point.
\end{proof}

Blowing up all nine base points of a pencil of cubics we obtain an elliptic fibration $f\colon \tilde X \to \IP^1$ with Euler number  $e(\tilde X) = 12\neq0$, so every pencil contains singular cubics. Since cubics with worse than nodal singularities form a subset of codimension at least two in the space of cubics, the general pencil will  contain  12 nodal cubics as singular elements. In particular, the general del Pezzo surface of degree 1 satisfies the assumptions made in  Construction \ref{Construction}.

We want to carry out the construction in from the previous section in a family: fix once for all coordinates $x,y,z$  in $\IP^2$ and let $\bar D_1= \{y^2z-x^2(x-z)=0\}$. Then let 
\[B' = \left\{(\bar D_2, p_0) \left| \begin{minipage}{5.8cm}
$\bar D_2$  nodal plane cubic, intersecting $\bar D_1$ transversely in smooth points,\\
$p_0\in \bar D_1\cap \bar D_2$
                                     \end{minipage}
\right.\right\}\subset \IP(\mathrm{Sym}^3\langle x,y,z\rangle)\times \bar D_1.\]
Looking at the projection to $\bar D_1$ we see that $B'$ is irreducible; the projection of $B'$ to the locus of nodal curves in $\IP(\mathrm{Sym}^3\langle x,y,z\rangle)$ has finite fibres, thus $\dim B'=8$.  Note that each pair $(\bar D_2, p_0)$ uniquely determines points $p_1, \dots, p_8$ such that $\bar D_1\cap \bar D_2 = \{ p_0, \dots p_8\}$.

Let    $\bar\gothD_1 =\bar D_1\times B'$ and let $\bar \gothD_2\subset\IP^2\times B'$ be the universal divisor given fibrewise by $\bar D_2$. We blow up $\IP^2\times B'$  in those components of $\bar \gothD_1\cap \bar \gothD_2$ that do not contain the point $p_0$ in any fibre. Passing to a smooth  open subset $B\subset B'$ and denoting the strict transforms of $\bar\gothD_i$ with the same symbol we get a flat family of pairs $\bar\phi\colon (\bar \gothX, \bar \gothD=\bar\gothD_1\cup\bar\gothD_2)\to B$
of smooth del Pezzo surfaces of degree  1 together with two different nodal cubics in the anti-canonical system. 

Possibly after a finite \'etale cover, which we omit from the notation, the normalisations $\bar \gothD_i$ are trivial $\IP^1$-bundles over $B$ and the fibrewise involution described in Construction \ref{Construction} extends to an involution on the family. By \cite[Cor.~5.33]{KollarSMMP}, there is a  family 
$\phi\colon\gothX \to B$
such that every fibre is a surface as constructed in the previous section. Note that the family of singular loci of the fibres $\gothD\to B$ is the trivial family, thus flat. Therefore in the exact sequence
\[
 0\to \pi_*\ko_{\bar\gothX}(-\bar\gothD)\to \ko_{\gothX} \to \ko_{\gothD}\to 0
\]
the first and the third sheaf are flat over $B$, so also $\phi$ is flat.

 We wrap up the discussion so far in the following statement:
\begin{prop}\label{prop: family}
 There are flat families 
\[\begin{tikzcd} \bar\gothX\arrow{dr}[swap]{\bar\phi} \rar  &\gothX\dar{\phi}\\&B\end{tikzcd}\]
over an irreducible  smooth base of dimension $8$ such that
\begin{enumerate}
 \item every fibre of $\phi$ is a Gorenstein stable Godeaux surface as described in Section \ref{sect: surface} and every general such surface occurs in the family,
\item $\bar\gothX\to \gothX$ is the simultaneous normalisation of the fibres of $\phi$, thus a family of del Pezzo surfaces of degree $1$, 
\item the family $\bar\phi\colon \bar\gothX\to B$ is universal in every point, thus $B$ admits a finite map to the moduli space of del Pezzo surfaces of degree 1. 
\end{enumerate}
\end{prop}
\begin{proof}
 We only need to show the last point. By Lemma \ref{lem: def dP} a deformation of  a smooth fibre $\bar X$ is uniquely determined by a deformation of the 8 points in the plane that we blow up. In our construction we assumed that the points are contained in $\bar D_1$, so any sufficiently small deformation of the points will still by contained in a nodal cubic. Under the action of  $\mathrm{PSL(3, \IC)}$ all nodal cubics are equivalent, so all deformations of $\bar X$ are in fact induced by moving the points on $\bar D_1$, which is what is parametrised by $B$. The universality now follows by smoothness and dimension reasons.
\end{proof}

\section{Infinitesimal deformations}\label{sect: deformations}

To understand the deformations of the surfaces constructed above we start in a slightly more general setting: let $X$ be a surface with at most local normal crossing singularities. We denote by $D$ the double locus in $X$, by $\bar X$ the normalisation and by $\bar D$ the preimage of $D$ in $\bar X$.  Then $\bar D$ is a nodal curve and the glueing induces an involution $\tau$ on $\bar D^\nu$, the normalisation of $\bar D$. The involution preserves preimages of nodes and we have a pushout diagram as in   \eqref{diagram: pushout}.

 We denote as usual\footnote{The surface  $X$ is a local complete intersection, even locally a hypersurface, so we can replace the cotangent complex by the sheaf of K\"ahler differentials.}
\[\kt^i_X = \shext^i(\Omega_X, \ko_X) \text{ and } T^i_X = \Ext^1(\Omega_X, \ko_X),\]
such that $\kt_X=\kt^0_X$ is the tangent sheaf of $X$ and  $T_X^1$ is the space of infinitesimal deformations of $X$. The local-to-global-$\Ext$-sequence (see \cite[Thm. 2.4.1]{SernesiDeformation}) reads as
\begin{equation}\label{eq: ext}
 0\to H^1(\kt_X)\to T^1_X\to H^0(\kt_X^1)\to H^2(\kt_X)\to T_X^2 \to H^1(\kt_X^1) \to 0
\end{equation}
because $X$ is a local complete intersection and thus $\kt^2_X=0$. 

To analyse the spaces in the above sequence we use some of  Friedman's results from  \cite{Friedman83}, which we will now recall.
The natural map $\Omega_X\to \pi_*\Omega_{\bar X}$ has as  image $\bar \Omega_X$, the double dual of $\Omega_X$, and as kernel the sheaf of torsion differentials $\theta_X$. By \cite[Prop.\ 1.5]{Friedman83} we have two short exact sequences
\begin{gather}
 0\to\theta_X\to \Omega_X\to \bar\Omega_X \to 0\label{seq1},\\
 0\to \bar\Omega_X\to \pi_*\Omega_{\bar X}\to \nu_*\bar\Omega_{D^\nu}\to 0.\label{seq2}
\end{gather}
Furthermore \cite[Prop.\ 1.10]{Friedman83} there are line bundles $\ko_D(-X)$ and $\ko_{\bar D}(-X)=\pi^*\ko_D(-X)$ such that
\begin{equation}\label{seq theta} 0\to \ko_D(-X)\to \pi_*\ko_{\bar D}(-X) \to \theta_X\to0\end{equation}
is exact. We need to compute the degree of these line bundles, so we look at local generators: Near a normal crossing point of $X$ with local equation $xy=0$ the sequence corresponds to 
\[ 0 \to \langle xy \rangle \overset d\to \langle xdy, ydx\rangle \to \theta_X\to 0\]
while at  the degenerate cusp with equation $xyz=0$ we have
\[0\to \langle xyz \rangle \overset d\to \langle xydz, xzdy, yzdx\rangle \to \theta_X\to 0\]
On one (local) component of the normalisation, say  $L_x=\{y=z=0\}$, the function $y$, resp.\ the form $dy$ is a generator for the conormal bundle of $L_1$ in $\{z=0\}$; writing $y=xy/x$ we see that a generator of the conormal bundle of $L_x$ is a generator of $\kn^*_{\{xy=0\}/\{z=0\}}\restr{L_x}\tensor \ko_{L_x}(P)$. 

By \cite[Prop.2.3]{Friedman83} the dual of $\ko_D(-X)$ is $\kt^1_X$ and thus by patching together the above local computation we get:
\begin{lem}\label{lem: degree T1}
 Let $\tau$ be the involution acting on $\bar D^\nu$, the normalisation of $\bar D$ and let $R_i\in \bar D^\nu$ be the preimages of nodes of $\bar D$. 
Then $\kt_X^1$ is locally free and $\nu^*\kt_X^1$ is the line bundle on $D^\nu$ corresponding the following  $\tau$-invariant line-bundle on $\bar D^\nu$:
\begin{align*}\pi^*\nu^*\kt^1_X&\isom\bar\nu^*\kn_{\bar D/\bar X}\left(-{\textstyle\sum}R_i\right) \tensor\tau^*\left(\bar\nu^*\kn_{\bar D/\bar X}\left(-{\textstyle\sum}  R_i\right)\right)\tensor \ko_{\bar D^\nu}\left({\textstyle\sum}  R_i\right)\\
&\isom \bar\nu^*\ko_{\bar X}(\bar D)\tensor\tau^*\bar\nu^*\ko_{\bar X}(\bar D)\tensor \ko_{\bar D^\nu}\left(-{\textstyle\sum}  R_i\right).
\end{align*}
\end{lem}

We now turn back  to our original example. 
\begin{prop}\label{prop: step 1}
In the situation of Construction \ref{Construction} the following hold true:
\begin{enumerate}
\item  The inclusion $\kt_{\bar X}(-\bar D)\into \kt_{\bar X}$ induces for all $i$ isomorphisms in cohomology $ H^i(\kt_{\bar X}(-\bar D))\isom H^i(\kt_{\bar X})$.
 \item $H^2(\kt_X)=0$ and the local-to-global-$\Ext$-sequence becomes 
 \[0\to T^1_{\bar X}\isom H^1(\kt_X)\to T^1_X\to H^0(\kt_X^1)\to 0 \text { and }  T_X^2 \isom  H^1(\kt_X^1) .\]
\item $h^0(\kt_X^1) = h^1(\kt_X^1)\leq 1$ and if both groups vanish then $X$ has unobstructed deformations and all deformations of $X$ are locally trivial deformations induced by a deformation of the del Pezzo surface $\bar X$. 
\end{enumerate}
\end{prop}
\begin{proof}
For \refenum{i}  we have to show that $\kt_X\restr{\bar D}$,  the cokernel of the inclusion,  has no cohomology.
Dualising the restriction sequence of K\"ahler differentials we get two exact sequences \cite[Prop.~1.1.8]{SernesiDeformation}
\begin{gather*}
 0\to \kt_{\bar D} \to \kt_{\bar X}\restr{\bar D}\to  \kn'\to 0,\\
0\to \kn'\to \ko_{\bar D}(\bar D) \to 
 \shext^1(\ko_{\bar D}, \ko_{\bar D})=\kt^1_{\bar D}\to 0.
\end{gather*}
We interpret the cohomology groups of these sequences in deformation theoretic terms \cite[Sect.~3]{SernesiDeformation}:
\begin{itemize}
 \item $H^0(\kt_{\bar D})$ is the tangent space to the automorphism group of $\bar D$, thus vanishes since $\bar D$ is stable. 
\item $H^1(\kt_{\bar D})$ is the space of infinitesimal locally trivial deformations  of $\bar D$, thus vanishes.
\item $\lambda\colon H^0(\ko_{\bar D}(\bar D)) \to H^0( \kt_{\bar D}^1)$ is the map that associates to an infinitesimal deformations of $\bar D$ inside $\bar X$ the local deformation at the nodes of $\bar D$. Both spaces have dimension $3$, the first by the restriction sequence and the second because $\bar D$ has three nodes. The map $\lambda$ is an isomorphism, because the nodes of $\bar D$ can be smoothed independently in $\bar X$: deforming  $\bar D_i$ to a smooth canonical curve smoothed $Q_i$ while keeping the other two nodes and deforming  $\bar D$ to a general member of $|-2K_{\bar X}|$ smoothed all nodes. 
\item $\ko_{\bar D}(\bar D)$ and $\kt^1_{\bar D}$ have no higher cohomology, thus by the second sequence and  the previous item $H^i(\kn')=0$ for all $i$. 
 \item $H^0(\kt_{\bar D})$ is the tangent space to the automorphism group of $\bar D$, thus vanishes since $\bar D$ is stable. 
\item $H^1(\kt_{\bar D})$ is the space of infinitesimal locally trivial deformations  of $\bar D$, thus vanishes.
\end{itemize}
Thus by the first sequence $H^i(\kt_{\bar X}\restr{\bar D})=0$ for all $i$ as claimed. 

We now turn to \refenum{ii}.
Since $\shom_{\ko_X}(\theta_X, \ko_X)=0$ the sequence \eqref{seq1} induces an isomorphism $\shom_{\ko_X}(\bar \Omega_X, \ko_X)\to \shom_{\ko_X}( \Omega_X, \ko_X)=\kt_X$. 

 Note that the relative dualising sheaf for $\pi $ is $\omega_\pi = \omega_{\bar X}\tensor \pi^*\inverse\omega_X\isom \ko_{\bar X}(-\bar D)$ and similarly the relative dualising sheaf for 
 $\nu\colon D^\nu \to X$ is $\omega_\nu \isom \ko_{D^\nu}(-3)$. 
 Hence by relative duality
\begin{gather*}
 \shext^i_{\ko_X}(\pi_*\Omega_{\bar X}, \ko_X)\isom \pi_*\shext^i_{\ko_{\bar X}}(\Omega_{\bar X}, \ko_{\bar X}(-\bar D))\isom \begin{cases} \pi_*\kt_{\bar X}(-\bar D) & i =0\\ 0 & i>0\end{cases},\\
  \shext^i_{\ko_X}(\nu_*\Omega_{D^\nu}, \ko_X)\isom \nu_*\shext^i_{\ko_X}(\ko_{D^\nu}(-2), \ko_{D^\nu}(-3))\isom  \begin{cases} \nu_*\ko_{D^\nu}(-1)& i =1\\ 0 & i\neq1\end{cases}
\end{gather*}
and thus applying $\shext_{\ko_X}(-, \ko_X)$ to \eqref{seq2} yields a short exact sequence
\[ 0\to \pi_*\kt_{\bar X}(-\bar D)\to \shom_{\ko_X}(\bar \Omega_X, \ko_X)\isom \kt_X\to \nu_*\ko_{D^\nu}(-1)\to 0.\]
In particular, we have induced isomorphisms
\begin{equation}\label{eq: cohomology of tangent sheaf}
  H^i(\kt_{\bar X}(-\bar D))\overset{\isom}{\longrightarrow} H^i(\kt_X).
\end{equation}
and the our claim follows from the first item.

For \refenum{iii} note that by  Lemma \ref{lem: degree T1} the line bundle $\kt_X^1$ on $D$ has degree $1=2\bar D.\bar D_1-3$. Since a line bundle on an irreducible curve of arithmetic genus $2$ can have at most one section, we get the first part of \refenum{ii}. If $H^0(\kt^1_X)=0$ then, by \refenum{i}, we have natural isomorphism  $T^1_{\bar X}\isom H^1(\kt_X)\isom T^1_X$ and $T^2_X=0$, and our claim  follows from \cite[Cor.~2.6.4]{SernesiDeformation}.
\end{proof}
Next we show that only the favourable case in the last item of the previous Proposition occurs.
\begin{lem}\label{lem: step 2}
 In the situation of Construction \ref{Construction} we have $h^0(\kt^1_X)=0$.
\end{lem}
\begin{proof}
Assume there is a non-zero section $\xi\in H^0(\kt^1_X)$. Since $\kt^1_X$ has degree 1 on $D$, the section $\xi$ vanishes exactly in one smooth point $R\in D$. 
By Proposition \ref{prop: step 1}\refenum{i} this section lifts to an infinitesimal deformation $\gothX\to B=\spec\left(\IC[\epsilon]/(\epsilon^2)\right)$ of $X$. Locally analytically along the singular locus of $X$, the family looks like the following models in $\IA^3_{B}$: 
\begin{align*}
&\text{at the degenerate cusp $P$:} && xyz+\epsilon\\
& \text{at a general nc point:} && xy+\epsilon\\
&\text{at the zero $R$ of $\xi$:} && xy+z\epsilon
\end{align*}
We will now derive a contradiction as follows: guided by the local models we  construct a partial-log-resolution of the central fibre $\sigma\colon\gothY \to \gothX$ with central fibre $Y$. Using the description of $\kt^1_{Y}$ from Lemma \ref{lem: degree T1} we then show that $Y$ cannot have a deformation which is non-singular in codimension 1. 

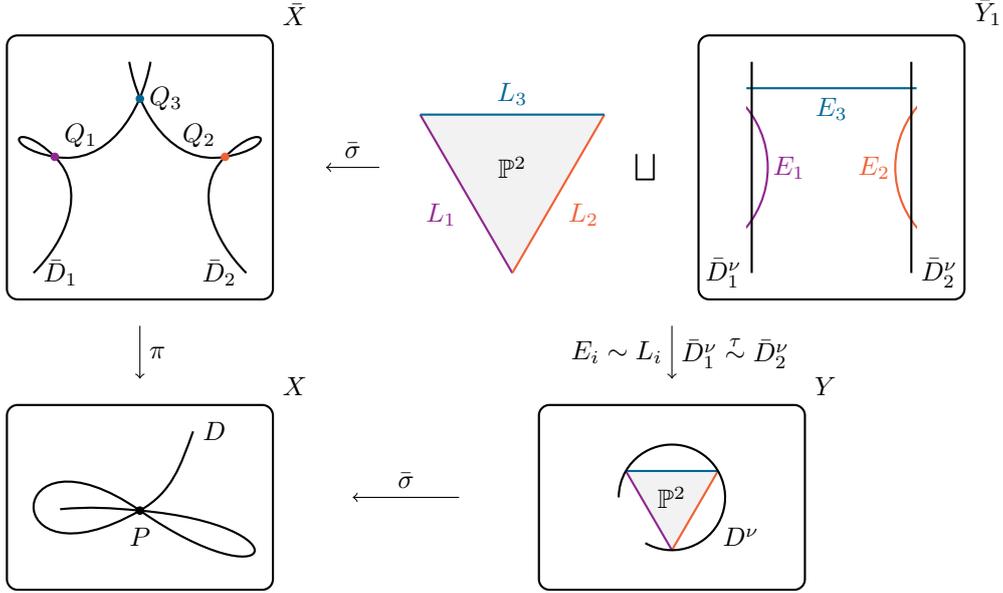
\begin{figure}
\small
\begin{tikzpicture}
[curve/.style ={thick, every loop/.style={looseness=10, min distance=30}},
q1/.style={color=Plum},
q2/.style={color=RedOrange},
q3/.style={color=MidnightBlue},
scale = 0.7
]

\begin{scope}
\node (Q3) at (0,2.3){};
\node (Q2) at (1.6,1.2) {};
\node (Q1) at (-1.6,1.2) {};

\draw[thick, rounded corners] (-2.5, -1.5) rectangle (2.5, 3.5) node [above right] {$\bar X$};
\draw [curve] (-2,-1) node [ right] {$\bar D_1$} to[out=45, in=315] (-1.6, 1.2) to[out=135, in =170,loop] () to[out=350, in = 260] (.2,3);
\draw [curve, xscale=-1] (-2,-1)  node [left] {$\bar D_2$} to[out=45, in=315] (-1.6, 1.2) to[out=135, in =170,loop] () to[out=350, in = 260] (.2,3);

\filldraw[q3] (Q3) circle (2pt) ; \node at (Q3) [right] {$Q_3$};
\filldraw[q2] (Q2) circle (2pt);  \node at (Q2) [above left] {$Q_2$};
\filldraw[q1] (Q1) circle (2pt); \node at (Q1) [above right] {$Q_1$};

\draw[->] (0,-2) to node[right] {$\pi$}++(0,-1);
\end{scope}

\begin{scope}[yshift=-5.5cm]
\draw[thick, rounded corners] (-2.5, -1.5) rectangle (2.5, 2) node [above right] {$X$};
\draw[curve,  every loop/.style={looseness=30, min distance=90}] (1,1.5) node [right]{$D$} to[out = 250, in = 30]
(0,0) to[out = 210, in = 145, loop]
() to[out = 325, in =355 ,  loop]
() to[out =175, in = 5] (179:1.5cm);

\filldraw (0,0) circle (2pt) node [below, yshift = -.1cm] {$P$};
\end{scope}

\begin{scope}[xshift = 7cm]
\begin{scope}[xshift = 6cm]
\draw[thick, rounded corners] (-2.5, -1.5) rectangle (2.5, 3.5) node [above right] {$\bar Y_1$};
\begin{scope}[curve, q1, yshift =.5cm]
\clip (-1.6,-2) rectangle (0, 2);
\node [draw] at (-3,.5) [circle through={(-1.5,-.5)}] {};
\node at (-.8, .5) {$E_1$};
\end{scope}

\begin{scope}[curve, q2,  yshift =.5cm]
\clip (1.6,-2) rectangle (0, 2);
\node [draw] at (3,.5) [circle through={(1.5,-.5)}] {};
\node at (.8, .5) {$E_2$};
\end{scope}

\draw [curve, q3] (-1.6,2.5) to node [ below] {$E_3$} ++(3.2,0);
\draw [curve] (-1.5,-1) node [ left] {$\bar D_1^\nu$}-- ++(0,4);
\draw [curve] (1.5,-1) node [ right] {$\bar D_2^\nu$}-- ++(0,4);
\end{scope}

\begin{scope}[xshift=0cm, yshift = 1cm]
\coordinate (a) at  (-90:2);
\coordinate (b) at  (30:2);
\coordinate (c) at  (150:2);

\fill[color=black!5] (a)--(b)--(c)-- cycle;
\draw[curve, q1] (a) -- node[below left] {$L_1$}(c);
\draw[curve, q2] (a) -- node[below right] {$L_2$}(b);
\draw[curve, q3] (b) --node[above] {$L_3$} (c);

\node at (0,0) {$\IP^2$};
\end{scope}

\node at (2.5,1){$\bigsqcup$};
\draw[->] (3,-2) to node[left] {$E_i\sim L_i$}node [right]{$\bar D^\nu_1\overset{\tau}\sim \bar D^\nu_2$} ++(0,-1);
\draw[->] (-2.5,1) to node[above] {$\bar\sigma$} ++(-1,0);

\end{scope}

\begin{scope}[xshift = 10cm, yshift = -5.25cm]
\draw[thick, rounded corners] (-2.5, -1.75) rectangle (2.5, 1.75) node [above right] {$Y$};

\coordinate (a) at  (-90:1);
\coordinate (b) at  (30:1);
\coordinate (c) at  (150:1);

\fill[color=black!5] (a)--(b)--(c)-- cycle;
\draw[curve, q1] (a) -- (c);
\draw[curve, q2] (a) -- (b);
\draw[ curve, q3] (b) --(c);
\node at (0,0) {$\IP^2$};

\draw[curve] (-120:1) arc (-120:180:1);
\node at (-30:1.5) {$D^\nu$};

\draw[->] (-4,0) to node[above] {$\bar\sigma$} ++(-2,0);

\end{scope}

\end{tikzpicture}
\caption{The partial log-resolution of the central fibre}\label{fig: log-resolution}
\end{figure}

So we blow up $P\in X\in \gothX$. From the local model we see that the normalisation of the central fibre, which is again a normal crossing divisor, consists of two components:
\[ \bar Y_1 = \mathrm{Bl}_{Q_1, Q_2, Q_3}\bar X \text{ and } \bar Y_2 = \IP^2\]
These are glued together as in Figure \ref{fig: log-resolution} and the singular locus has four irreducible components: the  curve $D^\nu$, which is the strict transform of the singular locus in $X$ and three curves $A_1,  A_2, A_3$ where we glue an exceptional curve of the blow up $\bar Y_1\to\bar X$ to a line in $\IP^2$. By Lemma \ref{lem: degree T1} we compute
\begin{gather*}
 \deg \kt_{Y}^1\restr{D_0} = -1-1-3 = -5\\
\deg \kt_{Y}^1\restr{A_i} = 1+3-2=2 \qquad(i = 1,2,3)
\end{gather*}
 Thus every section of $\kt_{Y}^1$ vanishes along $D_0$ which is a contradiction to the existence of the family $\gothY$ constructed.
\end{proof}
\begin{rem}
 I  would have liked to compute $H^0(\kt^1_X)$ directly in terms of the construction, instead of constructing the partial log-resolution. However, the precise way to nail down the glueing data on  $\nu^*\kt^1_X\isom \ko_{D^\nu}(1)$ which makes it a sheaf on $D$ evaded me. 

It seems plausible that some of the above can be formulated favourably in the language of log-structures but it is unclear to me if this gives a neater proof.
\end{rem}

\section{Proof of Theorem A}\label{sect: proof}
In Section \ref{sect: family} we constructed a family $\phi\colon \gothX \to B$ of Gorenstein stable surfaces over a smooth base which parametrises the general surface as constructed in Section \ref{sect: surface}. In addition, the family admits a simultaneous normalisation which is a family of del Pezzo surfaces of degree 1, universal in every point (see Proposition \ref{prop: family}).

We have proved Theorem $A$ as soon as we show that this family is locally at a point $b \in B$ the universal deformation space of the fibre $X=\gothX_b$. This is accomplished by Proposition \ref{prop: step 1}\refenum{iii} combined with Lemma \ref{lem: step 2}.

\section{The canonical ring in a particular example}\label{sect: 3}
One original motivation to study this example was the search for Gorenstein stable surfaces  with a complicated canonical ring and possibly the 5-canonical map not an embedding. We thought that this example was suitable, because the invariants are small and the double curve prevents the 4-canonical map from being an embedding.

More precisely, we proved in \cite[Example~47]{liu-rollenske14} that since $K_X$ has degree 1 one $D$  the 2-canonical map has base points and neither the 3-canonical nor the 4-canonical map are embeddings when restricted to $D$.

We will disprove these expectations in one particular example, that we will now describe. Our computation, which we will only sketch, is based on the following result of Koll\'ar, which we state  in a simplified version.

\begin{prop}[{\cite[Prop.~5.8]{KollarSMMP}}]\label{prop: sections}
Let $X$ be a Gorenstein stable  surface. Define the different  $\Delta = \Diff_{\bar D^\nu}(0)$ by the equality $(K_{\bar X}+\bar D)\restr{ \bar D}  = K_{\bar D}+ \Delta$. 

Then a section $s\in H^0(\bar X, m(K_{\bar X}+\bar D))$ descends to a section in $H^0(X, mK_X)$ if and only if the image of $s$ in $H^0(\bar D^\nu, m( K_{\bar D}+ \Delta))$ under the Residue map is $\tau$-invariant if $m$ is even respectively $\tau$-anti-invariant if $m$ is odd.
 \end{prop}
  
Let $S=\IC[x_1, x_2, y,z]$ be a polynomial ring with weights $(1,1,2,3)$ and let $f =   z^2+y^3-(x_1-x_2)yz+x_1^5x_2+x_1x_2^5$. Then
\[\bar X=\Proj(S/f)\subset\IP(1,1,2,3)\]
is a smooth del Pezzo surface of degree 1, $S/f$ is the anti-canonical ring of $\bar X$  and the equations $x_i=0$ define nodal rational curves $\bar D_i\in |-K_{\bar X}|$. We now construct a Gorenstein stable Godeaux surface as in Construction \ref{Construction}.

Since $K_{\bar X}+\bar D = -K_{\bar X}$ we can, by Proposition \ref{prop: sections},  consider the canonical ring of $X$ as a subring of the anti-canonical ring of $\bar X$.

In the notation of Figure \ref{fig: construction} we have 
\[K_{\bar D}+\Delta=(K_{\bar X}+\bar D)\restr{\bar D}  = K_{\bar D^\nu}+\sum_{ij} Q_{ij}\]
which has degree one on each component. We choose generators for  $R_{\bar D^\nu}$, the ring of sections of $K_{\bar D}+\Delta$, such that the sum of the  residue map in the single degrees is described algebraically by 
\[
\begin{split}
\mathrm{Res}\colon S/f \to&\,  R_{\bar D^\nu}\isom \IC[a_1, b_1]\times \IC[a_2,b_2]
\\
(x_1, x_2, y,z)\mapsto &\, (b_2-a_2, b_1-a_1,  -a_1b_1-a_2b_2,  -a_1^2b_1+a_2^2b_2) 
\end{split}
\]
Let $R^\tau$ be the subring of $R_{\bar D^\nu}$ consisting of $\tau$-anti-invariant sections in odd degrees and $\tau$-invariant sections in even degrees, that is, the pullback to $\bar D^\nu$ of the ring of sections of $K_X\restr D^\nu$. 
By Proposition \ref{prop: sections} the canonical ring of $X$ is  the preimage of $R^\tau$ in $S/f$. For a particular choice of $\tau$, we have 
\[ R^\tau = \IC[a_1+a_2, b_1+(a_2-b_2)]\subset R_{\bar D^\nu}.\]

Using this explicit model we computed the canonical ring of $X$ with  Singular \cite{singular}. We summarise the result in the following proposition.

\begin{prop}
 Let $X$ be the surface constructed from $\bar X = \Proj S/f$ as described in Construction \ref{Construction}. Then the following holds:
 \begin{enumerate}
  \item The canonical ring $R(X)$ is generated as a subring of $S/f$ by 
  \begin{gather*}
x_1^2+x_2^2-y,
x_1 x_2,\\
x_1 y+x_2 y,
x_1^3-x_2^3+3 x_2 y+z,
x_1 x_2^2,
x_1^2 x_2,\\
x_1 z-x_2 z,
x_1^2 y-x_2^2 y-x_2 z,
x_1 x_2 y,
x_1 x_2^3,\\
x_2 y^2-x_1^2 z,
x_1 y^2+x_2^2 z,
x_1 x_2 z, 
  \end{gather*}
  and there are 54 relation of in degrees $(6^6, 7^{12},  8^{18},9^{12}, 10^6)$.
  Consequently, there is an embedding into weighted projective space \[X\isom \Proj R(X)\into \IP(2^2,3^4, 4^4,5^3)\]
 as a subvariety of codimension ten.
  \item The bicanonical map $X\dashrightarrow \IP^1$ has a unique base point at $P$.
  \item The tricanonical map is a birational morphism onto the surface in $\IP^3$ of degree 9 with equation
\begin{align*}
 z_2^9&-5 z_0 z_2^7 z_3+2 z_1 z_2^7 z_3+4 z_2^8 z_3+6 z_0^2 z_2^5 z_3^2-5 z_0 z_1 z_2^5 z_3^2+z_1^2 z_2^5 z_3^2-11 z_0 z_2^6 z_3^2\\
& +6 z_1 z_2^6 z_3^2+6 z_2^7 z_3^2+z_0^3 z_2^3 z_3^3+9 z_0^2 z_2^4 z_3^3-11 z_0 z_1 z_2^4 z_3^3+3 z_1^2 z_2^4 z_3^3-7 z_0 z_2^5 z_3^3\\
& +6 z_1 z_2^5 z_3^3 +2 z_2^6 z_3^3+3 z_0^2 z_2^3 z_3^4-7 z_0 z_1 z_2^3 z_3^4+3 z_1^2 z_2^3 z_3^4+4 z_0 z_2^4 z_3^4-5 z_2^5 z_3^4\\
 &-z_0 z_1 z_2^2 z_3^5
+z_1^2 z_2^2 z_3^5+11 z_0 z_2^3 z_3^5-6 z_1 z_2^3 z_3^5-5 z_2^4 z_3^5+7 z_0 z_2^2 z_3^6-6 z_1 z_2^2 z_3^6\\
 &+2 z_2^3 z_3^6
 +z_0 z_2 z_3^7 -2 z_1 z_2 z_3^7+6 z_2^2 z_3^7+4 z_2 z_3^8+z_3^9 
\end{align*}
The singular locus consists of the two lines $z_2=z_3=0$ and $z_0=z_2+z_3=0$; the first line is the image of $D$. A real picture of the surface with the singular locus clearly visible is shown in Figure \ref{fig: X3}.
\begin{figure}\caption{The image of the 3-canonical map}\label{fig: X3}
 \includegraphics[width=4cm]{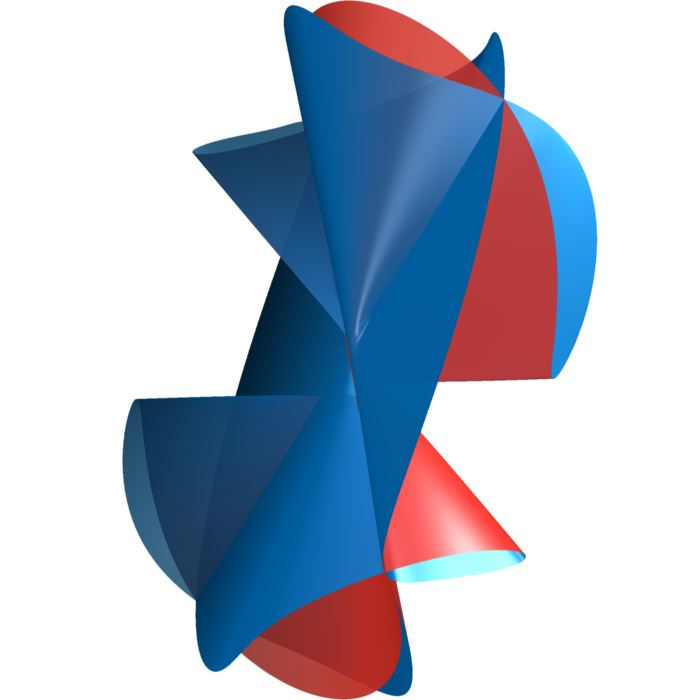}
\end{figure}

\item The 4-canonical map is a birational morphism onto a surface of degree 16 in $\IP^6$.
\item The 5-canonical map is an embedding. 
 \end{enumerate}
\end{prop}

\begin{rem}
 \begin{enumerate}
  \item A very successful method to study canonical rings is restriction to a canonical curve. Here we don't have a canonical curve but it is tempting to look at the restriction to the non-normal locus $D$. However, this does not help, because $D$ is not Cartier and in fact there is a generator in degree 5 that vanishes on $D$.
\item The properties of the pluricanonical maps in low degrees can easily be worked out directly without the help of a computer algebra system.
\item Appearantly, the number of relations in the ring suggests that it might be obtained by repeated unprojections but at this level of complexity, Gorenstein in codimension 10,  I have not even tried to guess a general format for the canonical rings of the surfaces in the family.
 \end{enumerate}
\end{rem}

%

\begin{thebibliography}{BHPV04}

\bibitem[BHPV04]{BHPV}
Wolf~P. Barth, Klaus Hulek, Chris A.~M. Peters, and Antonius {Van de Ven}.
\newblock {\em Compact complex surfaces}, volume~4 of {\em Ergebnisse der
  Mathematik und ihrer Grenzgebiete. 3. Folge.}
\newblock Springer-Verlag, Berlin, second edition, 2004.

\bibitem[Cat88]{catanese88}
Fabrizio Catanese.
\newblock Moduli of algebraic surfaces.
\newblock In {\em Theory of moduli (Montecatini Terme, 1985)}, volume 1337 of
  {\em Lecture Notes in Math.}, pages 1--83. Springer, Berlin, 1988.

\bibitem[DGPS10]{singular}
W.~Decker, G.-M. Greuel, G.~Pfister, and H.~Sch{\"o}nemann.
\newblock {\sc Singular} {3-1-2} --- {A} computer algebra system for polynomial
  computations.
\newblock 2010.
\newblock http://www.singular.uni-kl.de.

\bibitem[FPR14a]{fpr14b}
Marco Franciosi, Rita Pardini, and S{\"o}nke Rollenske.
\newblock Computing invariants of semi-log-canonical surfaces.
\newblock preprint arXiv:1404.3548, 2014.

\bibitem[FPR14b]{fpr14}
Marco Franciosi, Rita Pardini, and S{\"o}nke Rollenske.
\newblock Log-canonical pairs and {Gorenstein} stable surfaces with
  {$K^2_X=1$}.
\newblock preprint arXiv:1403.2159, 2014.

\bibitem[FR14]{franciosi-rollenske14}
Marco Franciosi and S{\"o}nke Rollenske.
\newblock Some remarks on pluricanonical maps of {G}orenstein stable surfaces.
\newblock unpublished manuscript, 2014.

\bibitem[Fri83]{Friedman83}
Robert Friedman.
\newblock Global smoothings of varieties with normal crossings.
\newblock {\em Ann. of Math. (2)}, 118(1):75--114, 1983.

\bibitem[Kol96]{KollarRationalCurves}
J{\'a}nos Koll{\'a}r.
\newblock {\em Rational curves on algebraic varieties}, volume~32 of {\em
  Ergebnisse der Mathematik und ihrer Grenzgebiete. 3. Folge.}
\newblock Springer-Verlag, Berlin, 1996.

\bibitem[Kol12]{kollar12}
Jan\'os Koll\'ar.
\newblock Moduli of varieties of general type.
\newblock In G.~Farkas and I.~Morrison, editors, {\em Handbook of Moduli:
  Volume II}, volume~24 of {\em Advanced Lectures in Mathematics}, pages
  131--158. International Press, 2012, arXiv:1008.0621.

\bibitem[Kol13]{KollarSMMP}
J{\'a}nos Koll{\'a}r.
\newblock {\em Singularities of the minimal model program}, volume 200 of {\em
  Cambridge Tracts in Mathematics}.
\newblock Cambridge University Press, Cambridge, 2013.
\newblock With a collaboration of S{\'a}ndor Kov{\'a}cs.

\bibitem[Kol14]{KollarModuli}
J{\'a}nos Koll{\'a}r.
\newblock {\em Moduli of varieties of general type}.
\newblock 2014.
\newblock book in preparation.

\bibitem[LP07]{lee-park07}
Yongnam Lee and Jongil Park.
\newblock {A simply connected surface of general type with $p\_g=0$ and
  $K^2=2$.}
\newblock {\em Invent. Math.}, 170(3):483--505, 2007, math/0609072.

\bibitem[LR14]{liu-rollenske14}
Wenfei Liu and S{\"o}nke Rollenske.
\newblock Pluricanonical maps of stable log surfaces.
\newblock {\em Adv. Math.}, 258:69--126, 2014.

\bibitem[Ser06]{SernesiDeformation}
Edoardo Sernesi.
\newblock {\em Deformations of algebraic schemes}, volume 334 of {\em
  Grundlehren der Mathematischen Wissenschaften}.
\newblock Springer-Verlag, Berlin, 2006.

\end{thebibliography}

\end{document}